\documentclass[11pt]{article}

\usepackage{amssymb}
\usepackage{amsmath}
\usepackage{color}
\usepackage{theorem}
\usepackage[all]{xy}

\newcommand{\XYMATRIX}{\xymatrix@M=6pt}
\newcommand{\aremb}{\ar@{^{(}->}}
\newcommand{\arembfrom}{\ar@{<-^{)}}}

\numberwithin{equation}{section}

\theorembodyfont{\slshape}

  \newtheorem{THM}{Theorem}[section]
  \newtheorem{LEM}[THM]{Lemma}
  \newtheorem{PROP}[THM]{Proposition}
  \newtheorem{COR}[THM]{Corollary}

\theorembodyfont{\rmfamily}

  \newtheorem{EX}{Example}[section]

\newif\ifQEDsign
\newcommand{\QED}{\global\QEDsigntrue\hfill$\square$}

\newenvironment{proof}%
    {\par\noindent\textit{Proof.}\global\QEDsignfalse}%
    {\ifQEDsign\else\QED\fi\par\bigskip\par}

\newcommand{\quotient}[2]{\genfrac{[}{]}{0pt}{}{#1}{#2}}

\renewcommand{\preceq}{\preccurlyeq}

\renewcommand{\le}{\leqslant}
\renewcommand{\ge}{\geqslant}
\renewcommand{\leq}{\leqslant}

\newcommand{\0}{\varnothing}

\renewcommand{\sec}{\cap}
\renewcommand{\phi}{\varphi}
\renewcommand{\epsilon}{\varepsilon}

\newcommand{\CC}{\mathbf{C}}
\newcommand{\DD}{\mathbf{D}}

\newcommand{\NN}{\mathbb{N}}
\newcommand{\PP}{\mathbf{P}}

\newcommand{\union}{\cup}

\newcommand{\Boxed}[1]{\mbox{$#1$}}

\newcommand{\id}{\mathrm{id}}
\newcommand{\ID}{\mathrm{ID}}
\newcommand{\Ob}{\mathrm{Ob}}

\newcommand{\Surj}{\mathrm{Surj}}

\newcommand{\op}{\mathrm{op}}

\newcommand{\im}{\mathrm{im}}
\newcommand{\Con}{\mathrm{Con}}
\newcommand{\nat}{\mathrm{nat}}
\newcommand{\FinSetInj}{\mathbf{FinSet}_{\mathit{inj}}}
\newcommand{\FinSetSurj}{\mathbf{FinSet}_{\mathit{surj}}}
\newcommand{\FinBaSurj}{\mathbf{FinBA}_{\mathit{surj}}}

\newcommand{\calA}{\mathcal{A}}
\newcommand{\calB}{\mathcal{B}}
\newcommand{\calC}{\mathcal{C}}

\newcommand{\calJ}{\mathcal{J}}
\newcommand{\calK}{\mathcal{K}}
\newcommand{\calL}{\mathcal{L}}
\newcommand{\calM}{\mathcal{M}}
\newcommand{\calN}{\mathcal{N}}
\newcommand{\calO}{\mathcal{O}}
\newcommand{\calP}{\mathcal{P}}
\newcommand{\calQ}{\mathcal{Q}}

\newcommand{\calS}{\mathcal{S}}

\newcommand{\Fraisse}{Fra\"\i ss\'e}

\DeclareMathOperator{\Aut}{Aut}

\title{On the Dual Ramsey Property for Finite Distributive Lattices}
\author{%
  Dragan Ma\v sulovi\'c, Neboj\v sa Mudrinski\\
  University of Novi Sad, Faculty of Sciences\\
  Department of Mathematics and Informatics\\
  Trg Dositeja Obradovi\'ca 3, 21000 Novi Sad, Serbia\\
  e-mail: $\{$dragan.masulovic,nebojsa.mudrinski$\}$@dmi.uns.ac.rs}

\begin{document}
\maketitle

\begin{abstract}
  The class of finite distributive
  lattices, as many other natural classes of structures,
  does not have the Ramsey property. It is quite common,
  though, that after expanding the structures with appropriatelly chosen linear orders
  the resulting class has the Ramsey property. So, one might expect
  that a similar result holds for the class of all finite distributive lattices.
  Surprisingly, Kechris and Soki\'c have proved in 2012 that this is not the case:
  no expansion of the class of finite distributive lattices by linear orders satisfies the
  Ramsey property.

  In this paper we prove that
  the class of finite distributive lattices does not have the dual Ramsey property either.
  However, we are able to derive a dual Ramsey
  theorem for finite distributive lattices endowed with a particular linear order.

  \bigskip


  \noindent \textbf{AMS Subj.\ Classification (2010):} 05C55, 06D50
\end{abstract}

\section{Introduction}

The class of finite distributive lattices does not have the Ramsey property (see~\cite{Promel-Voigt}),
and this is not an exception:
many natural classes of structures (such as finite graphs, metric spaces and posets, just to
name a few) do not have the Ramsey property. It is quite common,
though, that after expanding the structures under consideration
with appropriatelly chosen linear orders, the resulting class of expanded
structures has the Ramsey property. For example, the class of all finite
linearly ordered graphs $(V, E, \Boxed\le)$ where $(V, E)$ is a finite graph and $\le$ is
a linear order on the set $V$ of vertices of the graph has the Ramsey property~\cite{AH,Nesetril-Rodl}.
The same is true for metric spaces~\cite{Nesetril-metric}.
In case of finite posets we consider the class of all
finite linearly ordered posets $(P, \Boxed\preceq, \Boxed\le)$ where $(P, \Boxed\preceq)$ is
a finite poset and $\le$ is a linear order on $P$ which extends~$\preceq$~\cite{Nesetril-Rodl}.
So, one might expect
that a similar result holds for the class of all finite distributive lattices.
Surprisingly, this is not the case. In~\cite{kechris-sokic} the authors prove that
no expansion of the class of finite distributive lattices by linear orders satisfies the
Ramsey property.

Using techniques developed in~\cite{masulovic-ramsey} we prove in this paper that
the class of finite distributive lattices does not have the dual Ramsey property either.
However, using the same techniques we are able to derive a dual Ramsey
theorem for finite distributive lattices endowed with a particular linear order which we refer to as
the \emph{natural order}.

Our approach is based on employing categorical equivalence. It has recently been established \cite{masulovic-ramsey} that
Ramsey property (formulated in the language of category theory) is preserved under categorical equivalence
(see Theorem~\ref{cerp.thm.dual}). In particular if $\CC$ and $\DD$ are dually equivalent categories then
$\CC$ has the Ramsey property if and only if $\DD$ has the dual Ramsey property. (For precise
notions and formulation of the statement see Section~\ref{cerp.sec.prelim}). By Birkhoff duality,
the category of finite posets is dually equivalent to the category of finite distributive lattices.
Since the class of finite posets does not have the Ramsey property, it follows immediately that the class of finite
distributive lattices does not have the dual Ramsey property.

Along the lines of Birkhoff duality, we develop in Section~\ref{cerp.sec.duality}
a duality between the category $\DD'$ whose objects are
naturally ordered finite distributive lattices and whose morphisms are
surjective $\{0,1\}$-homomorphisms that respect the additional linear order in a particular way,
and the category $\PP'$ whose objects are finite posets with a linear order that extends the poset
ordering, and whose morphisms are poset embeddings that preserve the additional linear order.
Now, the class of finite posets with a linear extension \emph{does have} the Ramsey property,
whence follows the dual Ramsey property for the class of finite distributive lattices endowed with a
particularly chosen linear order. The duality between the two categories yields the dual Ramsey property
for naturally ordered finite distributive lattices in terms of morphisms. At the end of the section
we reformulate the property in terms of special partitions of lattices.

\section{Preliminaries}
\label{cerp.sec.prelim}

\paragraph{Categories.}
In order to specify a category $\CC$ one has to specify
a class of objects $\Ob(\CC)$,
a set of morphisms $\hom_\CC(\calA, \calB)$ for all $\calA, \calB \in \Ob(\CC)$,
an identity morphism $\id_\calA$ for all $\calA \in \Ob(\CC)$, and
the composition of mor\-phi\-sms~$\cdot$~so that
\begin{itemize}
\item $(f \cdot g) \cdot h = f \cdot (g \cdot h)$, and
\item $\id_\calB \cdot f = f \cdot \id_\calA$ for all $f \in \hom_\CC(\calA, \calB)$.
\end{itemize}
Let $\Aut(\calA)$ denote the set of all invertible morphisms $\hom_\CC(\calA, \calA)$.
Recall that an object $\calA \in \Ob(\CC)$ is \emph{rigid} if $\Aut(\calA) = \{\id_\calA\}$.

Note that morphisms in $\hom_\CC(\calA, \calB)$ are not necessarily structure-preserving mappings
from $\calA$ to $\calB$, and that the composition
$\cdot$ in a category is not necessarily composition of mappings.
Instead of $\hom_\CC(\calA, \calB)$ we write $\hom(\calA, \calB)$ whenever $\CC$ is obvious from the context.

For a category $\CC$, the oposite category, denoted by $\CC^\op$, is the category whose objects
are the objects of $\CC$, morphisms are formally reversed so that
$
  \hom_{\CC^\op}(A, B) = \hom_\CC(B, A),
$
and so is the composition:
$
  f \cdot_{\CC^\op} g = g \cdot_\CC f.
$

In this paper we consider categories of finite first-order structures.
Base sets of structures $\calA$, $\calA_1$, $\calA^*$, \ldots\ will always be denoted by
$A$, $A_1$, $A^*$, \ldots\ respectively.
Given a structure $\calA = (A, \Delta)$ and a linear order $<$,
we write $\calA_<$ for the structure $(A, \Delta, \Boxed<)$.
Moreover, we shall always write $\calA$ to denote the corresponding reduct of~$\calA_<$.
Linear orders denoted by $<$, $\sqsubset$ etc.\ are irreflexive (strict linear orders), whereas
by $\le$, $\sqsubseteq$ etc.\ we denote the corresponding reflexive linear orders.

\paragraph{Equivalent and dually equivalent categories.}
Categories $\CC$ and $\DD$ are \emph{isomorphic} if there exist functors $E : \CC \to \DD$ and $H : \DD \to \CC$ such that
$H$ is the inverse of $E$. A functor $E : \CC \to \DD$ is isomorphism-dense if
for every $D \in \Ob(\DD)$ there is a $C \in \Ob(\CC)$ such that $E(C) \cong D$.
Categories $\CC$ and $\DD$ are \emph{equivalent} if there exist functors $E : \CC \to \DD$
and $H : \DD \to \CC$, and natural isomorphisms $\eta : \ID_\CC \to HE$ and $\epsilon : \ID_\DD \to EH$.
We say that $H$ is a pseudoinverse of $E$ and vice versa.
It is a well known fact that a functor $E : \CC \to \DD$ has a pseudoinverse if and only if it is full, faithful and isomorphism-dense.
Clearly, if $E$ has a pseudoinverse then $\CC$ and $\DD$ are equivalent.
A \emph{skeleton} of a category is a full, isomorphism-dense subcategory in which no two distinct objects are isomorphic.
It is easy to see that (assuming (AC)) every category has a skeleton. It is also a well known fact that two categories are
equivalent if and only if they have isomorphic skeletons.
Categories $\CC$ and $\DD$ are \emph{dually equivalent} if $\CC$ and $\DD^\op$ are equivalent.

\paragraph{Ramsey property for categories.}
We say that
$
  \calS = \calM_1 \union \ldots \union \calM_k
$
is a \emph{$k$-coloring} of $\calS$ if $\calM_i \sec \calM_j = \0$ whenever $i \ne j$.
Given a category $\CC$, define $\sim_\calA$ on
$\hom_\CC(\calA, \calB)$ as follows: for $f, f' \in \hom_\CC(\calA, \calB)$
we let $f \sim_\calA f'$ if $f' = f \cdot \alpha$
for some $\alpha \in \Aut(\calA)$. Then
$$
  {\binom \calB \calA}_\CC = \hom_\CC(\calA, \calB) / \Boxed{\sim_\calA}
$$
corresponds to all subobjects of $\calB$ isomorphic to $\calA$ (see~\cite{Nesetril,NR-2}).
For an integer $k \ge 2$ and $\calA, \calB, \calC \in \Ob(\CC)$ we write
$$
  \calC \longrightarrow (\calB)^{\calA}_k
$$
to denote that for every $k$-coloring
$
  {\binom \calC\calA}_\CC = \calM_1 \union \ldots \union \calM_k
$
there is an $i \in \{1, \ldots, k\}$ and a morphism $w \in \hom_\CC(\calB, \calC)$ such that
$w \cdot {\binom \calB\calA}_\CC \subseteq \calM_i$.
(Note that $w \cdot (f / \Boxed{\sim_\calA}) = (w \cdot f) / \Boxed{\sim_\calA}$ for $f / \Boxed{\sim_\calA}
\in {\binom\calB\calA}_\CC$.)  We write
$$
  \calC \overset{\mathit{hom}}\longrightarrow (\calB)^{\calA}_k
$$
to denote that for every $k$-coloring
$
  \hom_\CC(\calA, \calC) = \calM_1 \union \ldots \union \calM_k
$
there is an $i \in \{1, \ldots, k\}$ and a morphism $w \in \hom_\CC(\calB, \calC)$ such that
$w \cdot \hom_\CC(\calA, \calB) \subseteq \calM_i$.

A category $\CC$ has the \emph{Ramsey property for objects} if
for every integer $k \ge 2$ and all $\calA, \calB \in \Ob(\CC)$
such that $\hom_\CC(\calA, \calB) \ne \0$ there is a
$\calC \in \Ob(\CC)$ such that $\calC \longrightarrow (\calB)^{\calA}_k$.
A category $\CC$ has the \emph{Ramsey property for morphisms} if
for every integer $k \ge 2$ and all $\calA, \calB \in \Ob(\CC)$
such that $\hom_\CC(\calA, \calB) \ne \0$ there is a
$\calC \in \Ob(\CC)$ such that $\calC \overset{\mathit{hom}}\longrightarrow (\calB)^{\calA}_k$.

In a category of finite ordered structures all the relations
$\sim_\calA$ are trivial and the two Ramsey properties introduced above coincide.
Therefore, we say that a category of finite ordered structures and embeddings has the \emph{Ramsey property}
if it has the Ramsey property for morphisms. (In general, the two notions do not coincide, but are
nevertheless closely related; see~\cite{zucker} and~\cite{masulovic-ramsey} for a short proof.)

\begin{EX}\label{cerp.ex.FRP}
    The category $\FinSetInj$ of finite sets and injective maps has the Ramsey property for objects.
    This is just a reformulation of the Finite Ramsey Theorem:

  \begin{THM}\label{cerp.thm.FRT} \cite{Ramsey}
    For all positive integers $k$, $a$, $m$ there is a positive integer $n$ such that
    for every $n$-element set $C$ and every $k$-coloring of the set $\binom Ca$ of all $a$-element subsets of $C$
    there is an $m$-element subset $B$ of $C$ such that $\binom Ba$ is monochromatic.
  \end{THM}
\end{EX}

A category $\CC$ has the \emph{dual Ramsey property for objects (morphisms)} if
$\CC^\op$ has the Ramsey property for objects (morphisms).

\begin{EX}\label{cerp.ex.FDRT}
    The category $\FinSetSurj$ of finite sets and surjective maps has the dual Ramsey property for objects.
    This is just a reformulation of the Finite Dual Ramsey Theorem:

  \begin{THM}\label{cerp.thm.FDRT} \cite{GR}
    For all positive integers $k$, $a$, $m$ there is a positive integer $n$ such that
    for every $n$-element set $C$ and every $k$-coloring of the set $\quotient Ca$ of all partitions of
    $C$ with exactly $a$ blocks there is a partition $\beta$ of $C$ with exactly $m$ blocks such that
    the set of all patitions from $\quotient Ca$ which are coarser than $\beta$ is monochromatic.
  \end{THM}

    To show that this is indeed the case, let $\CC = \FinSetSurj$. For $A, B \in \Ob(\CC)$
    let $\Surj(B, A)$ denote the set of all surjective maps $B \twoheadrightarrow A$.
    Define $\equiv_A$ on $\Surj(B, A)$ as follows: for $f, f' \in \Surj(B, A)$ we let
    $f \equiv_A f'$ if $f' = \alpha \circ f$ for some bijection~$\alpha : A \to A$.

    The claim that $\CC^\op$ has the Ramsey property for objects means that
    for every integer $k \ge 2$ and all $A, B \in \Ob(\CC)$
    such that $\hom_{\CC^\op}(A, B) \ne \0$ there is a $C \in \Ob(\CC)$ such that for every $k$-coloring
    $
      {\binom CA}_{\CC^\op} = \calM_1 \union \ldots \union \calM_k
    $
    there is an $i \in \{1, \ldots, k\}$ and a morphism $w \in \hom_{\CC^\op}(B, C)$ such that
    $w \cdot {\binom BA}_{\CC^\op} \subseteq \calM_i$. Since
    \begin{align*}
      f \cdot g \text{ (in $\CC^\op$)} &= g \circ f\\
      \hom_{\CC^\op}(A, B) &= \hom_\CC(B, A) = \Surj(B, A)\\
      {\binom BA}_{\CC^\op} &= \hom_{\CC^\op}(A, B) \text{ factored by } \sim_A \text{ in } \CC^\op\\
                                    &= \hom_\CC(B, A) \text{ factored by } \equiv_A \text{ in } \CC\\
                                    &= \Surj(B, A) / \Boxed{\equiv_A},
    \end{align*}
    the fact that $\CC^\op$ has the Ramsey property for objects reads as follows:
    for every integer $k \ge 2$ and all finite sets $A$ and $B$
    such that $\Surj(B, A) \ne \0$ there is a finite set
    $C$ such that for every $k$-coloring
    $$
      \Surj(C, A) / \Boxed{\equiv_A} \; = \; \calM_1 \union \ldots \union \calM_k
    $$
    there is an $i \in \{1, \ldots, k\}$ and a surjective mapping
    $w \in \Surj(C, B)$ satisfying
    $(\Surj(B, A) / \Boxed{\equiv_A}) \circ w \subseteq \calM_i$.
    (Note that $(f / \Boxed{\equiv_A}) \circ w = (f \circ w) / \Boxed{\equiv_A}$ for $f / \Boxed{\equiv_A} \in
    \Surj(B, A) / \equiv_A$.)
    Since $\Surj(B, A) / \Boxed{\equiv_A}$ corresponds to partitions of $B$ into $|A|$-many blocks
    we see that the categorical statement is indeed a reformulation of Theorem~\ref{cerp.thm.FDRT}.
\end{EX}

Our main tool to derive a new Ramsey-type result for
finite distributive lattices is the following result from \cite{masulovic-ramsey}.

\begin{THM}\label{cerp.thm.dual} \cite{masulovic-ramsey}
  Let $\CC$ and $\DD$ be equivalent categories.
  Then $\CC$ has the Ramsey property for objects (morphisms) if and only if $\DD$ does.

  In particular if $\CC$ and $\DD$ are dually equivalent then
  $\CC$ has the dual Ramsey property for objects (morphisms) if and only if $\DD$ has
  the Ramsey property for objects (morphisms).
\end{THM}

\begin{EX}
    The category $\FinSetInj$ of finite sets and injective maps is dually equivalent to the category
    $\FinBaSurj$ of finite boolean algebras and surjective homomorphisms (Stone duality).
    Since $\FinSetInj$ has the Ramsey property for objects (Example~\ref{cerp.ex.FRP}), it follows that
    the category $\FinBaSurj$ has the dual Ramsey property for objects.

    Let us make this statement explicit. Let $\CC = \FinBaSurj$. For $\calA, \calB \in \Ob(\CC)$
    let $\Surj(\calB, \calA)$ denote the set of all surjective homomorphisms $\calB \twoheadrightarrow \calA$.
    Define $\equiv_\calA$ on $\Surj(\calB, \calA)$ as follows: for $f, f' \in \Surj(\calB, \calA)$ we let
    $f \equiv_\calA f'$ if $f' = \alpha \circ f$ for some $\alpha \in \Aut(\calA)$.

    As in the Example~\ref{cerp.ex.FDRT},
    the fact that $\CC^\op$ has the Ramsey property for objects takes the following form:
    for every integer $k \ge 2$ and all finite boolean algebras $\calA$ and $\calB$
    such that $\Surj(\calB, \calA) \ne \0$ there is a finite boolean algebra
    $\calC$ such that for every $k$-coloring
    $$
      \Surj(\calC, \calA) / \Boxed{\equiv_\calA} \; = \; \calM_1 \union \ldots \union \calM_k
    $$
    there is an $i \in \{1, \ldots, k\}$ and a surjective homomorphism
    $w \in \Surj(\calC, \calB)$ satisfying
    $(\Surj(\calB, \calA) / \Boxed{\equiv_\calA}) \circ w \subseteq \calM_i$.
    Since $\Surj(\calB, \calA) / \Boxed{\equiv_\calA}$ corresponds to congruences $\Phi$ of $\calB$
    such that $\calB / \Phi \cong \calA$ the above statement
    can be reformulated as follows:
    \begin{quote}
      Let $\Con(\calB)$ denote the set of congruences of an algebra $\calB$, and
      for algebras $\calA$ and $\calB$ of the same type let
      $$
        \Con(\calB, \calA) = \{ \Phi \in \Con(\calB) : \calB / \Phi \cong \calA \}.
      $$
      For every finite bolean algebra $\calB$, every $\Phi \in \Con(\calB)$ and
      every $k \ge 2$ there is a finite boolean algebra $\calC$ such that for every
      $k$-coloring of $\Con(\calC, \calB / \Phi)$ there is a congruence
      $\Psi \in \Con(\calC, \calB)$ such that the set of all the congruences from
      $\Con(\calC, \calB / \Phi)$ which contain $\Psi$ is monochromatic.
    \end{quote}
\end{EX}

\paragraph{Birkhoff duality for finite distributive lattices.}
Let us recall some basic facts about the Birkhoff duality between the category $\DD$ of
finite distributive lattices with $\{0,1\}$-homomorphisms and the category $\PP$ of finite posets with poset homomorphisms.
For every finite distributive lattice $\calL$, let $J(\calL)$ be the set of all join-irreducible elemets of $\calL$ and let
$\calJ(\calL)$ denote the poset of its join irreducible elements.
On the other hand, for a poset $\calQ$ let $O(\calQ)$ denote the set of all down-sets of $\calQ$ and let
$\calO(\calQ)$ denote the lattice of down-sets of $\calQ$ with set-theoretic
union and intersection as lattice operations and $\0$ and $Q$ as the bottom and top element.
Then $\calJ : \DD \to \PP$ and $\calO : \PP \to \DD$ are contravariant functors whose behaviour on morphisms is
given as follows. For a $\{0,1\}$-lattice homomorphism $f : \calL \to \calK$ we have $\calJ(f) : \calJ(\calK) \to \calJ(\calL)$
where $\calJ(f)(y) = \bigwedge f^{-1}(\uparrow_\calK y)$. (Note that in case $f$ is surjective we have
$\calJ(f)(y) = \bigwedge f^{-1}(y)$.)
On the other hand, for a poset homomorphism
$\phi : \calP \to \calQ$ we have $\calO(\phi) : \calO(\calQ) \to \calO(\calP)$ where $\calO(\phi)(U) = \phi^{-1}(U)$.
Moreover, $\calJ$ and $\calO$ constitute a dual equivalence between $\DD$ and $\PP$ (see~\cite{davey-priestley} for details).
Recall that for every finite distributive lattice $\calL$ the natural isomorphism $\eta_\calL : \calL \to \calO(\calJ(\calL))$
is given by $x \mapsto \downarrow_{J(\calL)} x$. Analogously,
for every finite poset $\calQ$ the natural isomorphism $\epsilon_\calQ : \calQ \to \calJ(\calO(\calQ))$
is given by $x \mapsto \downarrow_{\calQ} x$.

\section{Birkhoff-type duality for naturally ordered finite distributive lattices}
\label{cerp.sec.duality}

As an immediate corollary of Theorem~\ref{cerp.thm.dual} we get that the class of finite
distributive lattices does not have the dual Ramsey property:

\begin{THM}
  The category $\DD_{\text{surj}}$ of finite distributive lattices and surjective
  lattice homomorphisms does not have the dual Ramsey property.
\end{THM}
\begin{proof}
  It is well known that the category $\PP_{\text{emb}}$ of finite posets and poset embeddings
  does not have the Ramsey property~\cite{Nesetril,sokic}.
  Therefore, by Theorem~\ref{cerp.thm.dual}, $\DD_{\text{surj}}$ cannot have the dual Ramesy
  property since $\PP_{\text{emb}}$ and $\DD_{\text{surj}}$ are dually
  equivalent (Birkhoff duality).
\end{proof}

Our aim in this section is to expand finite distributive lattices by appropriatelly chosen linear orders
and to derive the dual Ramsey theorem for such expanded lattices. The expansion we propose is
analogous to the expansion of finite boolean algebras by linear orders proposed in~\cite{KPT}.

Let $\calL = (L, \Boxed\lor, \Boxed\land, 0, 1)$ be a finite distributive lattice, let $<$ be the lattice order
and let $J(\calL) = \{j_1, j_2, \ldots, j_n\}$.
Recall that every lattice element $x$ has the unique representation
$x = \delta_1 \cdot j_{1} \lor \delta_2 \cdot j_{2} \lor \ldots \lor \delta_n \cdot j_{n}$
where $\delta_i \in \{0, 1\}$ satisfies $\delta_i = 1$ if and only if $j_i \le x$,
and with the convention that $0 \cdot z = 0$ while $1 \cdot z = z$, $z \in L$.
Let $\sqsubset^0$ be a linear order on $J(\calL)$ which extends $<$ (that is, $j < j' \Rightarrow
j \mathrel{\sqsubset^0} j'$ for all $j, j' \in J(\calL)$). For the sake of notation let
$j_1 \mathrel{\sqsubset^0} j_2 \mathrel{\sqsubset^0} \ldots \mathrel{\sqsubset^0} j_n$.
The linear ordering of the join-irreducible elements of $\calL$ now induces a linear order on the whole of $\calL$
as follows. Take any $x, y \in L$, and let $x = \delta_1 \cdot j_{1} \lor \delta_2 \cdot j_{2} \lor \ldots \lor \delta_n \cdot j_{n}$
and $y = \epsilon_1 \cdot j_{1} \lor \epsilon_2 \cdot j_{2} \lor \ldots \lor \epsilon_n \cdot j_{n}$
be the unique representations of $x$ and $y$, respectively, where $\epsilon_s, \delta_s \in \{0, 1\}$.
We then say that $x \mathrel{\sqsubset} y$ if there is an $s$ such that $\delta_s < \epsilon_s$,
and $\delta_t = \epsilon_t$ for all $t > s$. In other words, $\sqsubset$ is the
\emph{antilexicographic ordering of the elements of $\calL$} with respect to~$\sqsubset^0$.

A linear ordering $\sqsubset$ of a finite distributive lattice $\calL$ will be referred to as
\emph{natural} if there is a linear ordering $\sqsubset^0$ on $J(\calL)$ which extends the lattice ordering
on $J(\calL)$ and $\sqsubset$ is the antilexicographic ordering of the elements of $\calL$ with respect to~$\sqsubset^0$.
It is easy to see that every natural linear ordering extends the lattice order $<$ on the entire $\calL$:

\begin{LEM}
  The natural linear ordering of a finite distributive lattice is a linear extension of the lattice ordering.
\end{LEM}
\begin{proof}
  Let $\calL$ be a finite distributive lattice with the lattice order~$<$ and
  let $\sqsubset$ be a natural linear order on $\calL$. For the sake of notation let
  $J(\calL) = \{j_1 \mathrel{\sqsubset} j_2 \mathrel{\sqsubset} \ldots \mathrel{\sqsubset} j_n\}$.
  Take any $x, y \in L$ such that $x < y$ and let us show that $x \sqsubset y$. This trivially holds if
  $x = 0$ because $0$ is the bottom element with respect to both $<$ and $\sqsubset$.
  Assume, therefore, that $x > 0$.

  Let $x = \delta_1 \cdot j_{1} \lor \delta_2 \cdot j_{2} \lor \ldots \lor \delta_n \cdot j_{n}$
  and $y = \epsilon_1 \cdot j_{1} \lor \epsilon_2 \cdot j_{2} \lor \ldots \lor \epsilon_n \cdot j_{n}$
  be the unique representations of $x$ and $y$, respectively, where $\epsilon_s, \delta_s \in \{0, 1\}$.
  Since $x < y$ there is a $j \in J(\calL)$ such that $j \le y$
  and $j \not\le x$, so let $s = \max\{i \in \{1, \ldots, n\} : j_i \le y \text{ and } j_i \not\le x\}$.
  Then $\delta_t = \epsilon_t$ for all $t > s$, while $\delta_s = 0 < 1 = \epsilon_s$.
  This concludes the proof that $x \sqsubset y$.
\end{proof}

Let $\calL$ be a finite distributive lattice and $\sqsubset$ a natural linear order on $\calL$.
The structure $(\calL, \Boxed\sqsubset)$ will shortly be denoted by~$\calL_\sqsubset$.
Analogously, if $\calP$ is a poset and $\prec$ is a linear order extending the ordering relation
of~$\calP$, the structure $(\calP, \Boxed\prec)$ will be referred to as
a \emph{linearly ordered poset} and denoted by~$\calP_\prec$.
Note that for a linearly ordered poset $\calP_\prec$ the linear order $\prec$
uniquely extends from $\calP$, which are the join-irreducibles of $\calO(\calP)$, to an
antilexicographic ordering of $\calO(\calP)$. This extension is a natural linear order for $\calO(\calP)$
and will be denoted by~$\calO(\calP)_\prec$. On the other hand, when restricted to the set
$J(\calL)$ of join-irreducibles of $\calL$, the natural linear order $\sqsubset$ of $\calL$ extends the
ordering relation of the poset $\calJ(\calL)$. The resulting linearly ordered poset will be denoted by $\calJ(\calL)_\sqsubset$.

Our aim now is to extend the Birkhoff duality between $\DD$ and $\PP$ to a duality between the
category of naturally ordered finite distributive lattices with special surjective lattice homomorphisms
referred to as \emph{positive} (to be defined soon), and the category of
finite linearly ordered posets and embeddings. As natural linear orders are closely related to the
lattice whose order they extend it should come as no surprise that the duality we are to
develop heavily relies on the Birhkoff duality. Since positive surjective lattice homomorphisms
\emph{are} lattice homomorphisms, we can dualise them \'a la Birkhoff, and then it turns out that
the Birkhoff dual is actually an embedding between the corresponding linearly ordered posets.
Dually, since embeddings of finite linearly ordered posets \emph{are} embeddings of finite posets,
we can dualise them \'a la Birkhoff, and it again turns out that
the Birkhoff dual is a positive surjective lattice homomorphism.

Let $f : \calL \to \calK$ be a lattice homomorphism. (In this setting every lattice homomorphism is a
$\{0,1\}$-homomorphism because we include the top and the bottom element of the lattice in the signature.)
Let
$$
  N(f) = \{x \in L : (\exists y \in L)(y < x \land f(y) = f(x)) \}.
$$
For an element $x \in L$ and a set of elements $S \subseteq L$, let
$$
  x - S = \bigvee \{j \in J(\calL) : j \le x \land j \notin S \}.
$$

\begin{LEM}\label{cerp.lem.1}
  Let $f : \calL \to \calK$ be a lattice homomorphism and let $\phi = \calJ(f) : \calJ(\calK) \to \calJ(\calL)$,\
  where $\calL$ and $\calK$ are finite distributive lattices. Take any $j \in J(\calL)$.

  $(a)$ $j \in \im(\phi)$ if and only if $j \notin N(f)$.

  $(b)$ Assume that $f$ is surjective (then $\phi$ is an embedding) and let $x \in L$. If $j \le x - N(f)$
  then $j \in \im(\phi)$ and $\phi^{-1}(j) = f(j)$.
\end{LEM}
\begin{proof}
  $(a)$
  Assume that $j \notin \im(\phi)$ and let $\downarrow_{J(\calL)} = \{j, k_1, \ldots, k_n\}$ where
  $j, k_1, \ldots, k_n$ and pairwise distinct join-irreducibles of $\calL$. Then, clearly,
  $\downarrow_{J(\calL)} j \supsetneqq \downarrow_{J(\calL)} \{k_1, \ldots, k_n\}$ while
  $\phi^{-1}(\downarrow_{J(\calL)} j) = \phi^{-1}(\downarrow_{J(\calL)} \{k_1, \ldots, k_n\})$
  (because $j \notin \im(\phi)$). Therefore, $j > k_1 \lor \ldots \lor k_n$ and $f(j) = f(k_1 \lor \ldots \lor k_n)$
  in $\calL$, whence $j \in N(f)$.

  Assume now that $j \in N(f)$. Because $\eta_\calL : \calL \to \calO(\calJ(\calL)) : x \mapsto \downarrow_{J(\calL)} x$
  is an isomorphism and because $\calO(\calJ(f)) = \phi^{-1}$ we have that $\downarrow_{J(\calL)} j \in N(\phi^{-1})$.
  Therefore, there is a $U \in O(\calL)$ such that $\downarrow_{J(\calL)} j \supsetneqq U$ and
  $\phi^{-1}(\downarrow_{J(\calL)} j) = \phi^{-1}(U)$. Consequently, $\phi^{-1}(j) = \0$ so
  $j \notin \im(\phi)$.

  $(b)$
  From $j \le x - N(f)$ we know that $j \notin N(f)$, so $j \in \im(\phi)$ by~$(a)$.
  Since $\phi$ is an embedding, it easily follows that $\phi^{-1}(\downarrow_{J(\calL)} j) =
  \downarrow_{J(\calK)} \phi^{-1}(j)$. On the other hand, the fact that $\eta$ is a natural
  transformation yields $\eta_\calK \circ f = \phi^{-1} \circ \eta_\calL$ whence
  $\downarrow_{J(\calK)} f(j) = \phi^{-1}(\downarrow_{J(\calL)} j)$. Finally,
  $\downarrow_{J(\calK)} f(j) = \downarrow_{J(\calK)} \phi^{-1}(j)$ whence
  $f(j) = \phi^{-1}(j)$ because $\eta_\calK$ is bijective.
\end{proof}

\begin{LEM}\label{cerp.lem.MinusN}
  Let $f : \calL \to \calK$ be a lattice homomorphism
  where $\calL$ and $\calK$ are finite distributive lattices, and let $x \in L$.
  Then $f(x) = f(x - N(f))$.
\end{LEM}
\begin{proof}
  Let $J(\calL) = \{j_1, j_2, \ldots, j_n\}$, $n\in\NN$, and let $x\in L$ be arbitrary.
  Let $x = \epsilon_1 \cdot j_1 \lor \epsilon_2 \cdot j_2 \lor \ldots \lor \epsilon_n \cdot j_n$
  be the unique representation of~$x$ (where $\epsilon_i \in \{0, 1\}$).
  Assume that $j_i \in N(f)$ for some $i$ satisfying $\epsilon_i = 1$, say,
  $x = j_1 \lor \epsilon_2 \cdot j_2 \lor \ldots \lor \epsilon_n \cdot j_n$ and $j_1 \in N(f)$.
  Let $x' = \epsilon_2 \cdot j_2 \lor \ldots \lor \epsilon_n \cdot j_n$ so that
  $x = x' \lor j_1$. By the assumption, $j_1 \in N(f)$ so there is a
  $t \in L$ such that $f(t)=f(j_1)$ and $t < j_1$.
  Let $t = \delta_1 \cdot j_1 \lor \delta_2 \cdot j_2 \lor \ldots \lor \delta_n \cdot j_n$
  be the unique representation of~$t$ (where $\delta_i \in \{0, 1\}$). Since $t < j_1$ we have that $\delta_1 = 0$, so
  $t = \delta_2 \cdot j_2 \lor \ldots \lor \delta_n \cdot j_n$. On the other hand, from $t < j_1 \le x$ it follows that
  $\delta_i = 1 \Rightarrow \epsilon_i = 1$ for all $i$. Hence, $t \le x'$.
  So, $f(x)=f(x'\vee j_1)=f(x')\vee f(j_1)=f(x')\vee f(t)=f(x'\vee t)=f(x')$. By iterating the same argument
  we end up with an $x^{(n)} \in L$ such that $f(x) = f(x') = \ldots = f(x^{(n)})$ and
  $x^{(n)} = x - N(f)$.
\end{proof}

Let $\calL_\sqsubset$ and $\calK_\prec$ be finite distributive lattices,
each with a natural linear order (indicated in the subscript).
We say that a lattice homomorphism $f : \calL \to \calK$ is \emph{positive} if
$$
  x - N(f) \sqsubseteq y - N(f) \Rightarrow f(x) \preceq f(y)
$$
for all $x, y \in L$.

\begin{LEM}\label{cerp.lem.tech}
  Let $\calL_\sqsubset$ and $\calK_\prec$ be naturally ordered finite distributive lattices.
  If $f : \calL_\sqsubset \twoheadrightarrow \calK_\prec$ is a positive surjective lattice homomorphism then
  $\calJ(f)$ is an embedding of the corresponding linearly ordered posets, that is,
  $\calJ(f) : \calJ(\calK)_\prec \hookrightarrow \calJ(\calL)_\sqsubset$.

  Dually, if $\phi : \calP_\sqsubset \hookrightarrow \calQ_\prec$ is an embedding of linearly ordered posets
  then $\calO(\phi)$ is a positive surjective lattice homomorphism
  $\calO(\calQ)_\prec \twoheadrightarrow \calO(\calP)_\sqsubset$.
\end{LEM}
\begin{proof}
  Assume that $f$ is a positive surjective lattice homomorphism and let
  $\phi = \calJ(f)$. Clearly, $\phi$ is an embedding $\calJ(\calK) \hookrightarrow \calJ(\calL)$,
  so we have to show that $\phi$ preserves and reflects the linear order, that is:
  $$
    j \preceq j' \text{ if and only if } \phi(j) \sqsubseteq \phi(j')
  $$
  for all $j, j' \in J(\calK)$. Since both $\sqsubseteq$ and $\preceq$ are linear orders, it suffices
  to show the implication from right to left.

  Assume that $\phi(j) \sqsubseteq \phi(j')$. Since $\phi(j), \phi(j') \in \im(\phi)$ and
  $\phi(j), \phi(j') \in J(\calL)$ it follows that $\phi(j) - N(f) = \phi(j)$ and
  $\phi(j') - N(f) = \phi(j')$. Therefore, $\phi(j) \sqsubseteq \phi(j')$ implies
  $\phi(j) - N(f) \sqsubseteq \phi(j') - N(f)$, so $f(\phi(j)) \preceq f(\phi(j'))$
  because $f$ is a positive homomorphism. Therefore, $j \preceq j'$ because
  $j = f(\phi(j))$ and $j' = f(\phi(j'))$ by Lemma~\ref{cerp.lem.1}~$(b)$.

  For the other part of the statement, assume that $\phi : \calP_\sqsubset \hookrightarrow \calQ_\prec$ is an embedding
  and let $f = \calO(\phi)$. Clearly, $f$ is a surjective lattice homomorphism
  $\calO(\calQ) \twoheadrightarrow \calO(\calP)$, so we have to show that $f$ is positive. Take $x, y \in \calO(\calQ)$
  and assume that $x - N(f) \prec y - N(f)$. Let $j_1, \ldots, j_n, j'_1, \ldots, j'_s, j''_1, \ldots, j''_t
  \in J(\calO(\calQ)) \setminus N(f)$ be join-irreducibles such that
  $x - N(f) = j'_1 \lor \ldots \lor j'_s \lor j_1 \lor \ldots \lor j_n$,
  $y - N(f) = j''_1 \lor \ldots \lor j''_t \lor j_1 \lor \ldots \lor j_n$,
  $j'_1 \prec \ldots \prec j'_s \prec j_1 \prec \ldots \prec j_n$,
  $j''_1 \prec \ldots \prec j''_t \prec j_1 \prec \ldots \prec j_n$ and $j'_s \prec j''_t$.
  Then
  \begin{align*}
    f(x) &= f(x - N(f)) \qquad \text{[Lemma~\ref{cerp.lem.MinusN}]}\\
         &= f(j'_1 \lor \ldots \lor j'_s \lor j_1 \lor \ldots \lor j_n)\\
         &= f(j'_1) \lor \ldots \lor f(j'_s) \lor f(j_1) \lor \ldots \lor f(j_n)\\
         &= \phi^{-1}(j'_1) \lor \ldots \lor \phi^{-1}(j'_s) \lor \phi^{-1}(j_1) \lor \ldots \lor \phi^{-1}(j_n) \qquad \text{[Lemma~\ref{cerp.lem.1}]}.
  \end{align*}
  Analogously,
  $$
    f(y) = \phi^{-1}(j''_1) \lor \ldots \lor \phi^{-1}(j''_t) \lor \phi^{-1}(j_1) \lor \ldots \lor \phi^{-1}(j_n).
  $$
  Since $\phi$ is an embedding, it follows that
  $\phi^{-1}(j'_1) \sqsubset \ldots \sqsubset \phi^{-1}(j'_s) \sqsubset \phi^{-1}(j_1) \sqsubset \ldots \sqsubset \phi^{-1}(j_n)$,
  $\phi^{-1}(j''_1) \sqsubset \ldots \sqsubset \phi^{-1}(j''_t) \sqsubset \phi^{-1}(j_1) \sqsubset \ldots \sqsubset \phi^{-1}(j_n)$
  and $\phi^{-1}(j'_s) \sqsubset \phi^{-1}(j''_t)$. Therefore, $f(x) \sqsubset f(y)$.
\end{proof}

\begin{PROP}
  $(a)$ Every isomorphism between naturally ordered finite distributive lattices
  is a positive homomorphism. In particular, the identity mapping is a positive isomorphism.

  $(b)$ The composition of positive surjective homomorphisms between naturally ordered finite distributive lattices
  is a positive surjective homomorphism.
\end{PROP}
\begin{proof}
  $(a)$ Easy, since $N(f) = \0$ for an isomorphism $f$.

  $(b)$ Let $\calL_\sqsubset$, $\calK_\prec$ and $\calM_<$ be naturally ordered finite distributive lattices
  and let $f : \calL_\sqsubset \to \calK_\prec$ and $g : \calK_\prec \to \calM_<$ be positive surjective
  homomorphisms.

  Let us first show the claim in case one of $f$, $g$ is an isomorphism. If $g$ is an isomorphism then
  $N(g \circ f) = N(f)$ and the claim follows directly. If $f$ is an isomorphism then it is easy
  to see that $N(g \circ f) = f^{-1}(N(g))$ and that $f(x - S) = f(x) - f(S)$ for all $x \in L$, $S \subseteq L$
  whence the claim follows by straightforward calculation.

  Let us now move on to the general case. Since $f$ and $g$ are positive surjective homomorphisms,
  $\calJ(f)$ and $\calJ(g)$ are embeddigs of the corresponding linearly ordered posets (Lemma~\ref{cerp.lem.tech}).
  Then $\calJ(g) \circ \calJ(f)$ is an embedding. Note that $\calJ(g) \circ \calJ(f) = \calJ(g \circ f)$ because
  $\calJ$ is a functor. Therefore, $\calJ(g \circ f)$ is an embedding of the corresponding linearly ordered posets,
  so $\calO(\calJ(g \circ f))$ is a positive surjective homomorphism (Lemma~\ref{cerp.lem.tech}).

  Recall that $\eta_\calL : \calL \to \calO(\calJ(\calL)) : x \mapsto \downarrow_{J(\calL)} x$ is an isomorphism
  for every finite distributive lattice $\calL$. It is easy to see that the same
  $\eta_\calL : x \mapsto \downarrow_{J(\calL)} x$ is an isomorphism
  $\calL_\sqsubset \to \calO(\calJ(\calL))_\sqsubset$, where $\calO(\calJ(\calL))_\sqsubset$
  denotes the finite distributive lattice $\calO(\calJ(\calL))$ where
  the natural linear order $\sqsubset$ of $\calL$ first
  restricts to $\calJ(\calL)$ and then extends uniquely to a natural linear order of~$\calO(\calJ(\calL))$.
  Since $\eta : \ID \to \calO \calJ$ is a natural transformation,
  $$
    g \circ f = \eta_\calM^{-1} \circ \calO(\calJ(g \circ f)) \circ \eta_\calL.
  $$
  As we have just seen, $\eta_\calM^{-1}$ and $\eta_\calL$ are isomorphisms of the corresponding naturally
  ordered finite distributive lattices and hence positive by~$(a)$,
  and $\calO(\calJ(g \circ f))$ is a positive surjective homomorphism. By the special case considered at
  the beginning of the proof, the composition $\eta_\calM^{-1} \circ \calO(\calJ(g \circ f)) \circ \eta_\calL$
  is also a positive surjective homomorphism, and this equals~$g \circ f$.
\end{proof}

Since the composition of positive surjective homomorphisms between
naturally ordered finite distributive lattices
is a positive surjective homomorphism, naturally ordered finite distributive lattices
together with positive surjective homomorphisms form a category with we denote by~$\DD'$.
Let~$\PP'$ be the category whose objects are finite linearly ordered posets
and whose morphisms are poset embeddings that preserve the additional linear order.

Let us define contravariant functors $\calJ' : \DD' \to \PP'$ and $\calO' : \PP' \to \DD'$ on objects
as follows: $\calJ'(\calL_\prec) = \calJ(\calL)_{\prec}$,
$\calO'(\calQ_{\Boxed{\sqsubset}}) = \calO(\calQ)_\sqsubset$;
and on morphisms as follows: for a positive surjective lattice homomorphism
$f : \calL_\sqsubset \twoheadrightarrow \calK_\prec$ we let $\calJ'(f) = \calJ(f)$,
and for an embedding $\phi : \calP_\sqsubset \hookrightarrow \calQ_\prec$ we let $\calO'(\phi) = \calO(\phi)$.

\begin{THM}\label{cerp.thm.dualeqiv}
  $\calJ'$ and $\calO'$ constitute a dual equivalence between $\DD'$ and $\PP'$.
\end{THM}
\begin{proof}
  Recall again that
  $\eta_\calL : \calL \to \calO(\calJ(\calL)) : x \mapsto \downarrow_{J(\calL)} x$ is an isomorphism
  for every finite distributive lattice $\calL$, that
  $\epsilon_\calQ : \calQ \to \calJ(\calO(\calQ)) : x \mapsto \downarrow_{\calQ} x$ is an isomorphism
  for every finite poset $\calQ$, and that the corresponding transformations $\eta : \ID \to \calO \calJ$
  and $\epsilon : \ID \to \calJ \calO$ are natural. Then it is easy to see that
  $\eta'_{\calL_\sqsubset} : \calL_\sqsubset \to \calO'(\calJ'(\calL_\sqsubset)) :
  x \mapsto \downarrow_{J(\calL)} x$ and
  $\epsilon'_{\calQ_\sqsubset} : \calQ_\sqsubset \to \calJ'(\calO'(\calQ_\sqsubset)) :
  x \mapsto \downarrow_{\calQ} x$ are indeed isomorphisms, and that the naturality of
  the transformations $\eta' : \ID \to \calO' \calJ'$ and $\epsilon' : \ID \to \calJ' \calO'$
  immediately follows from the fact that $\eta$ and $\epsilon$ are natural
  because $\calJ'(f) = \calJ(f)$ and $\calO'(\phi) = \calO(\phi)$.
\end{proof}

\begin{COR}\label{rcl.cor.catram}
  The category $\DD'$ has the dual Ramsey property.
\end{COR}
\begin{proof}
  Theorem~\ref{cerp.thm.dual} implies that the category
  $\DD'$ has the dual Ramsey property because the category $\PP'$ has the Ramsey property~\cite{sokic}
  and $(\DD')^\op$ and $\PP'$ are equivalent by Theorem~\ref{cerp.thm.dualeqiv}.
\end{proof}

Let us now spell out this result in terms of special partitions of lattices.

Let $\calL_\sqsubset$ be a naturally ordered finite distributive lattice and let
$\Phi$ be a congruence of $\calL$. Let
$$
  N(\Phi) = \{x \in L : (\exists y \in L)(y < x \land (x, y) \in \Phi) \}.
$$
We say that $\Phi$ is a \emph{positive congruence of $\calL_\sqsubset$} if the
following holds for all pairs $A, B$ of distinct congruence classes of $\Phi$:
\begin{gather*}
  (\forall a \in A)(\forall b \in B) a - N(\Phi) \sqsubseteq b - N(\Phi)
  \text{\quad or}\\
  (\forall a \in A)(\forall b \in B) a - N(\Phi) \sqsupseteq b - N(\Phi).
\end{gather*}

The following is a well known fact in lattice theory but we nevertheless include the
proof to keep the paper self contained.

\begin{LEM}\label{classic}
  Let $\calL$ be a finite distributive lattice and $\Phi$ a congruence
  of $\calL$. Then the smallest element of each join irreducible class
  in $\calL/\Phi$ is join irreducible in $\calL$.
\end{LEM}
\begin{proof}
  Let $[a]_{\Phi}\in J(\calL/\Phi)$ such that $a$ is the smallest
  element in $[a]_{\Phi}$. If $a=b\vee c$ for $b,c\in L$, then
  $[a]_{\Phi}=[b]_{\Phi}\vee[c]_{\Phi}$. By the assumption we have
  $[a]_{\Phi}=[b]_{\Phi}$ or $[a]_{\Phi}=[c]_{\Phi}$. Suppose that
  $[a]_{\Phi}=[b]_{\Phi}$. Then $b\in[a]_{\Phi}$. Hence $a\leq b$, but
  from $a=b\vee c$ we obtain $b\leq a$. Therefore,~$a=b$.
\end{proof}

\begin{LEM}\label{factor}
  Let $\calL_\sqsubset$ be a naturally ordered finite distributive
  lattice and let $\Phi$ be a positive congruence of
  $\calL_\sqsubset$. If we define a linear order $\preceq$ on $\calL/\Phi$ by
  $A\preceq B$ if and only if
  $(\forall a \in A)(\forall b \in B) a -N(\Phi) \sqsubseteq b -
  N(\Phi)$ for every two classes $A$ and $B$ of the congruence $\Phi$, then
  $\preceq$ is a natural linear ordering that extends the lattice
  ordering on $\calL/\Phi$.
\end{LEM}
\begin{proof}
  Note, first, that if $[i]_{\Phi},[j]_{\Phi}\in J(\calL/\Phi)$
  are such that $i = \min [i]_{\Phi}$ and $j= \min [j]_{\Phi}$ then $i,j\in J(\calL)$ by
  Lemma~\ref{classic} and we have that $i = i - N(\Phi)$ and $j = j - N(\Phi)$.
  Now, $i\sqsubseteq j\Leftrightarrow i-N(\Phi)\sqsubseteq
  j-N(\Phi)\Leftrightarrow[i]_\Phi\preceq[j]_\Phi$ by the
  definition of $\preceq$ and positivity of $\Phi$. Also,
  $i\leq j\Leftrightarrow[i]_{\Phi}\leq[j]_{\Phi}$.

  In order to show that $\preceq$ is a natural linear order, let us first show that it
  extends $\le$ on $J(\calL/\Phi)$. Take any
  $[i]_{\Phi},[j]_{\Phi}\in J(\calL/\Phi)$ such that $i = \min [i]_{\Phi}$ and $j= \min [j]_{\Phi}$.
  If $[i]_{\Phi}\leq[j]_{\Phi}$ then $i \le j$, so $i \sqsubseteq j$ because $\sqsubseteq$
  extends $\le$ in $\calL$; finally, $i \sqsubseteq j$ implies $[i]_\Phi\preceq[j]_\Phi$ as we have seen at the
  beginning of the proof.
  %  If $i\leq j$ then $i\wedge j=i$
  %  and using homomorphism property for $nat_{\Phi}$ we obtain $[i]_{\Phi}\leq[j]_{\Phi}$. If
  %  $[i]_{\Phi}\leq[j]_{\Phi}$ then $[i]_{\Phi}=[i]_{\Phi}\wedge[j]_{\Phi}=[i\wedge j]_{\Phi}$.
  %  Hence $i\wedge j\in[i]_{\Phi}$ whence $i\leq i\wedge j$. Therefore $i=i\wedge j$ or equivalently $i\leq j$.

  Finally, let us show that $\preceq$ is the antilexicographic ordering of $\calL/\Phi$ with respect
  to the restriction of~$\preceq$ on~$J(\calL/\Phi)$. Let $J(\calL/\Phi) = \{[j_1]_\Phi, \ldots, [j_n]_\Phi\}$
  where $j_i = \min[j_i]_\Phi$ for all $i$. Without loss of generailty we can assume that
  $j_1 \sqsubseteq j_2 \sqsubseteq \ldots \sqsubseteq j_n$. Then, as we have seen at the beginning of the proof,
  $[j_1]_\Phi \preceq [j_2]_\Phi \preceq \ldots \preceq [j_n]_\Phi$, and Lemma~\ref{classic} yields
  $j_1, \ldots, j_n \in J(\calL)$. Take any $[x]_\Phi, [y]_\Phi \in \calL/\Phi$ and let
  $
    [x]_\Phi = \bigvee_i \epsilon_i \cdot [j_i]_\Phi
  $
  and
  $
    [y]_\Phi = \bigvee_i \delta_i \cdot [j_i]_\Phi
  $
  be the unique representations of $[x]_\Phi$ and $[y]_\Phi$ in $\calL/\Phi$ (where $\epsilon_i, \delta_i \in \{0,1\}$
  for all~$i$). Assume now that $[x]_\Phi \preceq [y]_\Phi$. Then
  $
    [x]_\Phi = \left[ \bigvee_i \epsilon_i \cdot j_i \right]_\Phi \preceq
    \left[ \bigvee_i \delta_i \cdot j_i \right]_\Phi = [y]_\Phi,
  $
  so
  $
    \left(\bigvee_i \epsilon_i \cdot j_i\right) - N(\Phi) \sqsubseteq \left(\bigvee_i \delta_i \cdot j_i\right) - N(\Phi).
  $
  Because of $j_i = \min[j_i]_\Phi$ for all $i$ we have that $\{j_1, \ldots, j_n\} \sec N(\Phi) = \0$, whence
  $\left(\bigvee_i \epsilon_i \cdot j_i\right) - N(\Phi) = \bigvee_i \epsilon_i \cdot j_i$ and
  $\left(\bigvee_i \delta_i \cdot j_i\right) - N(\Phi) = \bigvee_i \delta_i \cdot j_i$. Therefore,
  $\bigvee_i \epsilon_i \cdot j_i \sqsubseteq \bigvee_i \delta_i \cdot j_i$, so there is an $s$ such that
  $\epsilon_s < \delta_s$ and $\epsilon_i = \delta_i$ for all $i > s$.
  (It may happen that $\{j_1, \ldots, j_n\}$ is not the entire $J(\calL)$, but then we note that
  the join irreducibles $j \in J(\calL) \setminus \{j_1, \ldots, j_n\}$ do not appear in
  $\bigvee_i \epsilon_i \cdot j_i$, or more precisely, appear in this join as $0 \cdot j$, which does not
  affect the conclusion.)
\end{proof}

\begin{LEM}\label{rcl.lem.poscongr}
  $(a)$ The kernel of a positive homomorphism is a positive congruence.

  $(b)$ Every positive congruence is the kernel of a positive homomorphism.
\end{LEM}
\begin{proof}
  (a) Let $\calL_\sqsubset$ and $\calK_\prec$ be finite distributive
  lattices and let $f:\calL\to\calK$ be a positive lattice
  homomorphism. Now, we have $N(\ker f)=\{x\in L : (\exists y\in L)(y<x \land
  f(x)=f(y))\}=N(f)$. Let $A$ and $B$ be two distinct
  classes of the congruence $\ker f$ and let $a\in A$ and $b\in B$.
  Since $\sqsubseteq$ is a linear order we have $a-N(\ker f)\sqsubseteq
  b-N(\ker f)$ or $b-N(\ker f)\sqsubseteq a-N(\ker f)$. Without loss of
  generality we can assume that $a-N(\ker f)\sqsubseteq b-N(\ker f)$. Then we have
  $a-N(f)\sqsubseteq b-N(f)$ and therefore $f(a)\preceq
  f(b)$, because $f$ is positive. Now, let $c\in A$ and $d\in B$. Our
  goal is to show that $c-N(\ker f)\sqsubseteq d-N(\ker f)$. Seeking a
  contradiction, suppose that $d-N(\ker f)\sqsubseteq c-N(\ker f)$. This
  yields $d-N(f)\sqsubseteq c-N(f)$. Then $f(d)\preceq f(c)$. From
  $c,a\in A$ and $d,b\in B$ we know $(c, a) \in \ker f$ and $(d, b) \in \ker f$.
  Hence, we have $f(c)=f(a)$ and $f(d)=f(b)$, whence $f(b)\preceq
  f(a)$. Therefore, $f(a)=f(b)$, whence $a,b\in A\sec B=\0$. Contradiction.

  (b) Let $\Phi$ be a positive congruence of $\calL_\sqsubset$. We
  know that $\Phi$ is the kernel of the natural homomorphism
  $\nat_{\Phi}:\calL \to \calL/\Phi : x \mapsto [x]_\Phi$.
  Let us prove that $\nat_{\Phi}$ is positive. Note first
  that $N(\nat_{\Phi})=\{x\in L : (\exists y\in L)(y<x \land
  \nat_{\Phi}(x)=\nat_{\Phi}(y))\}=\{x\in L : (\exists y\in
  L)(y<x \land [x]_{\Phi}=[y]_{\Phi})\}=\{x\in L : (\exists y\in
  L)(y<x \land (x,y)\in\Phi)\}=N(\Phi)$. Now, we take $a,b\in L$ such
  that $a-N(\nat_{\Phi})\sqsubseteq b-N(\nat_{\Phi})$. Then
  $a-N(\Phi)\sqsubseteq b-N(\Phi)$. By positivity of $\Phi$
  we obtain $[a]_\Phi\preceq[b]_\Phi$, where $\preceq$ is defined in
  Lemma~\ref{factor} and it is a natural linear ordering of $\calL/\Phi$
  that extends the lattice ordering of $J(\calL/\Phi)$. Therefore
  $\nat_{\Phi}(a)\preceq \nat_{\Phi}(b)$.
\end{proof}

Finally, we are ready to present the dual Ramsey property for naturally ordered finite distributive lattices.

\begin{THM}
    Let $\Con^+(\calL_\sqsubset)$ denote the set of positive congruences of a naturally ordered finite distributive
    lattice $\calL_\sqsubset$, and let
    $$
      \Con^+(\calL_\sqsubset, \calK_\prec) = \{ \Phi \in \Con^+(\calL_\sqsubset) :
      \calL_\sqsubset / \Phi \cong \calK_\prec \}.
    $$
    For every naturally ordered finite distributive lattice $\calL_\sqsubset$,
    every $\Phi \in \Con^+(\calL_\sqsubset)$ and
    every $k \ge 2$ there is a naturally ordered finite distributive lattice $\calN_\triangleleft$
    such that for every $k$-coloring of $\Con^+(\calN_\triangleleft, \calL_\sqsubset / \Phi)$
    there is a congruence $\Psi \in \Con^+(\calN_\triangleleft, \calL_\sqsubset)$
    such that the set of all the congruences from
    $\Con^+(\calN_\triangleleft, \calL_\sqsubset / \Phi)$ which contain $\Psi$ is monochromatic.
\end{THM}
\begin{proof}
  The category $\DD'$ has the dual Ramsey property by Corollary~\ref{rcl.cor.catram}.
  Let us make this statement explicit having in mind that naturally ordered finite distributive lattices
  are rigid (have trivial automorphism groups).
  Let $\Surj^+(\calL_\sqsubset, \calK_\prec)$ denote the set of all positive surjective
  homomorphisms $\calL_\sqsubset \twoheadrightarrow \calK_\prec$. Then
  the fact that $\DD'$ has the dual Ramsey property means the following:
    for every integer $k \ge 2$ and all naturally ordered finite distributive lattices
    $\calL_\sqsubset$ and $\calK_\prec$ such that $\calL_\sqsubset \twoheadrightarrow \calK_\prec$
    there is a naturally ordered finite distributive lattice $\calN_\triangleleft$
    such that for every $k$-coloring
    $$
      \Surj^+(\calN_\triangleleft, \calK_\prec) = \calM_1 \union \ldots \union \calM_k
    $$
    there is an $i \in \{1, \ldots, k\}$ and a positive surjective homomorphism
    $w : \calN_\triangleleft \twoheadrightarrow \calL_\sqsubset$ satisfying
    $\Surj^+(\calL_\sqsubset, \calK_\prec) \circ w \subseteq \calM_i$.
    Since $\Surj^+(\calL_\sqsubset, \calK_\prec)$ corresponds to positive congruences $\Phi$ of $\calL_\sqsubset$
    such that $\calL_\sqsubset / \Phi \cong \calK_\prec$ (Lemma~\ref{rcl.lem.poscongr}), the statement follows.
\end{proof}


\begin{thebibliography}{99}
\bibitem{AH}
  F.\ G.\ Abramson, L.\ A.\ Harrington.
  Models without indiscernibles.
  J.~Symbolic Logic 43 (1978), 572--600.

\bibitem{davey-priestley}
  B.\ A.\ Davey, H.\ A.\ Priestley.
  Introduction to Lattices and Order (2nd Ed).
  Cambridge University Press, 2002.

\bibitem{GR}
  R.\ L.\ Graham, B.\ L.\ Rothschild.
  Ramsey's theorem for n-parameter sets.
  Tran.\ Amer.\ Math.\ Soc.\ 159 (1971), 257--292.

%\bibitem{GRS}
%  R.\ L.\ Graham, B.\ L.\ Rothschild, J.\ H.\ Spencer.
%  Ramsey Theory (2nd Ed).
%  John Wiley \& Sons, 1990.
%
%\bibitem{Fouche}
%  W.\ Fouch\'e.
%  Symmetries in Ramsey theory.
%  East-West Journal of Math., 1(1) (1998), 43--60.

\bibitem{KPT}
  A.\ S.\ Kechris, V.\ G.\ Pestov, S.\ Todor\v cevi\'c.
  \Fraisse\ limits, Ramsey theory and topological dynamics of automorphism groups.
  GAFA Geometric and Functional Analysis, 15 (2005) 106--189.

\bibitem{kechris-sokic}
  A.\ Kechris, M.\ Soki\'c.
  Dynamical properties of the automorphism groups of the random poset and random distributive lattice.
  Fund.\ Math.\ 218 (2012), 69--94.

\bibitem{masulovic-ramsey}
  D.\ Ma\v sulovi\'c, L.\ Scow.
  Categorical Constructions and Ramsey Property.
  arXiv:1506.01221

\bibitem{Nesetril}
  J.\ Ne\v set\v ril.
  Ramsey classes and homogeneous structures.
  Combinatorics, probability and computing, 14 (2005) 171--189.

\bibitem{Nesetril-metric}
  J.\ Ne\v set\v ril.
  Metric spaces are Ramsey.
  European Journal of Combinatorics 28 (2007), 457--468.

\bibitem{Nesetril-Rodl}
  J.\ Ne\v set\v ril, V.\ R\"odl.
  Partitions of finite relational and set systems.
  J.\ Combin.\ Theory Ser.\ A 22 (1978), 289--312.

\bibitem{NR-2}
  J.\ Ne\v set\v ril, V.\ R\"odl.
  Partition (Ramsey) theory and its applications.
  In: Surveys in Combinatorics, Cambridge University Press, Cambridge (1979), 96--156.

\bibitem{Promel-Voigt}
  H.\ J.\ Pr\"omel, B.\ Voigt.
  Recent results in partition (Ramsey) theory for finite lattices.
  Discr.\ Math., 35 (1981), 185--198.

\bibitem{Ramsey}
  F.\ P.\ Ramsey.
  On a problem of formal logic.
  Proc.\ London Math.\ Soc.\ 30 (1930), 264--286.

\bibitem{sokic}
  M.\ Soki\'c.
  Ramsey Properties of Finite Posets.
  Order, 29 (2012) 1--30.

\bibitem{zucker}
  A.\ Zucker.
  Topological dynamics of closed subgroups of $S_\infty$.
  arXiv:1404.5057.

\end{thebibliography}
\end{document}